\documentclass[11pt]{article}
\usepackage{amsfonts,amsmath,amssymb,amsthm, amscd, subcaption}
\usepackage{graphicx}
\usepackage[stable]{footmisc}
\numberwithin{equation}{section}

\newtheorem{theorem}{Theorem}[section]

\newtheorem{lemma}[theorem]{Lemma}
\newtheorem{cor}[theorem]{Corollary}

\theoremstyle{definition}
\newtheorem{definition}[theorem]{Definition}
\newtheorem{example}[theorem]{Example}
\newtheorem{remark}[theorem]{Remark}

\def\<{{\langle}}
\def\>{{\rangle}}
\def\G{{\Gamma}}
\def\a{{\alpha}}
\def\b{{\beta}}

\def\Z{\mathbb Z}

\def\S{{\mathbb S}}

\def\a{\alpha}
\def\b{\beta}

\def\t{\tau}

\def\G{\Gamma}

\def\e{\epsilon}

\def\Si{{\Sigma}}

\def\ni{\noindent} 

\begin{document}

\title{Goeritz and Seifert Matrices from Dehn Presentations} 

\author{Daniel S. Silver
\and
Lorenzo Traldi
\and Susan G. Williams \thanks {The first and third authors are partially supported by the Simons Foundation. 
} }

\maketitle 


\begin{abstract} The Goeritz matrix of a link is obtained from the Jacobian matrix of a modified Dehn presentation associated to a diagram using Fox's free differential calculus. When the diagram is special the Seifert matrix can also be determined from the presentation. \end{abstract}

\begin{center} Keywords: \textit {link, Dehn presentation, Goeritz matrix, Seifert matrix};   AMS Classification: MSC: 57M25 \end{center}

\section{Introduction} \label{Intro} Lebrecht Goeritz introduced an integral matrix associated to knot diagrams in \cite{Go33}. The Goeritz matrix represents a class of quadratic forms of a link, a class that is invariant under link isotopy. Goeritz showed that the $p$th Minkowski units of his matrix, for $p$ an odd prime, are invariant. So too is the absolute value of the determinant of the Goeritz matrix.

Herbert Seifert soon recognized the topological importance of the Goeritz matrix \cite{Se36}. It is a relation matrix for the first homology group of the 2-fold cover of the 3-sphere $\S^3$ branched over each component of the link, and it determines a linking form on homology classes. 

During the succeeding half century, the definition of the Goeritz matrix was extended and modified (see \cite{GL78, Ky54, Mu65, Tr85}). However, unlike the Alexander matrix (see below), which can be derived from a presentation of the link group, the Goeritz matrix has been defined by combinatorial and topological means. Our purpose is to show how the Goeritz arises directly from a presentation of the link group that is closely related to the well-known Dehn presentation, using the machinery of Fox's free differential calculus. 

The presentation that we give (Theorem \ref{main}) is obtained from a link diagram. When 
the diagram is special and not split, the presentation yields a Seifert matrix for the link (Theorem \ref{seifert}). Consequently,  link invariants that are obtainable from a Seifert matrix can also be found from such a group presentation. For example, it is well known that the Blanchfield pairing of a knot is such an invariant. (See \cite{FP17} for a new proof of this fact and some of its history.) 

The present paper was motivated by \cite{SW18}, in which the first and third authors showed that a Seifert  matrix of an oriented link arises as the Laplacian matrix of a directed checkerboard graph with $\pm 1$-edge weights of a special diagram of the link. 

\section{The Goeritz matrix}  
\label{goeritz}

Consider a link $L \subset \S^3$ with diagram $D$ in the plane. It is well known that the complementary regions of $D$ can be shaded in checkerboard fashion so that each arc of the diagram separates a shaded region from an unshaded one. We adopt the common convention that the unbounded region is unshaded.



The \textit{reduced Goeritz matrix} $G=G(D)$ associated with the checkerboard diagram $D$ is an $n \times n$ integral matrix, where $n$ is the number of bounded unshaded regions. Let $U_0$ be the unbounded region, let $U_1, \ldots, U_n$ be the bounded unshaded regions, and, for $i \ne j \in \{0,\ldots,n\}$, let $C_{i,j}$ be the set of crossings of $D$ at which both $U_i$ and $U_j$ are incident. Let $\eta(c) = \pm 1$ denote the \textit{Goeritz index} of $c$, as in Figure \ref{index}. Then for $1 \leq i,j \leq n$ the $i,j$th entry $G_{i,j}$ is given by \begin{equation} G_{i,j} = \left\{
	\begin{array}{ll}

		-\sum\limits_{c \in C_{i,j}} \eta(c)  & \mbox{if } i \ne j \\
		\sum\limits_{\substack{k=0 \\ k\neq i}}^{n} \;\; \sum\limits_{c \in C_{i,k}} \eta(c) & \mbox{if } i =j. 
	\end{array} \right.  \end{equation}
	
\begin{figure}
\begin{center}
\includegraphics[height=1.2 in]{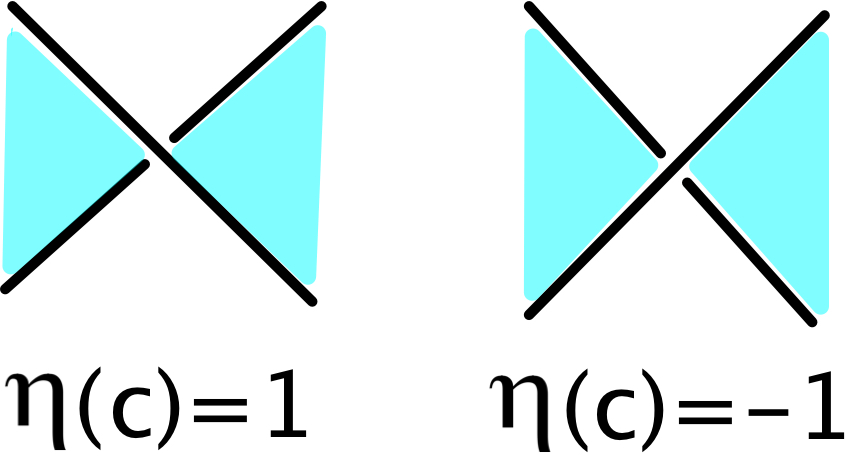}
\caption{Goeritz index of crossing $c$}
\label{index}
\end{center}
\end{figure}

We will make frequent use of the \textit{checkerboard graph} $\G=\G(D)$ with vertices corresponding to the shaded regions of $D$. Two vertices are joined by an edge whenever the corresponding regions meet at a crossing. (If only one shaded region appears at a crossing, the corresponding edge is a loop.) For notational convenience, we use the symbols of the shaded regions of $D$ also for the vertices of $\G$. We label each edge of $\G$ with \textit{weight} $+1$ or $-1$ according to the Goeritz index of the corresponding crossing  (Figure \ref{index}). The number of connected components of $\G$ will be denoted by $\b=\b(D)$.

The reduced Goeritz matrix $G$ is a reduced version of the much-studied Laplacian matrix associated with a checkerboard graph defined in the same way, but with a vertex for each unshaded region of $D$ instead. 

\section{Fox's differential calculus} \label{Fox} 

The free differential calculus \cite{Fo53, Fo54, Fo62, CF63} is a standard tool in both knot theory and combinatorial group theory. We briefly review the fundamental ideas. 

\begin{definition} \label{fdc} Let $w=w_1^{\e_1}\cdots w_n^{\e_n}$ be a word in symbols $w_1, \ldots, w_n$ with 
exponents $\e_i \in \{1, -1\}$. The symbols $w_1, \ldots, w_n$ need not be distinct. For each $i$, let $W_i$ denote the initial subword $w_1^{\e_1}\cdots w_{i-1}^{\e_{i-1}}$. For $w \in \{w_1, \ldots, w_n\}$, the \textit{partial derivative} $\partial W/\partial w$ is the element of the integral group ring of the free group on $w_1, \ldots, w_n$:
\begin{equation} \frac{\partial W}{\partial w} = \sum_{w_i = w} \left\{
	\begin{array}{ll}

		W_i  & \mbox{if } \e_i=1\\
		-W_iw^{-1} & \mbox{if } \e_i=-1. 
	\end{array} \right.  \end{equation}
	(The empty subword $W_1$ is indentified with the identity element $1$ of the group ring.)
	\end{definition}

\begin{definition} Let $\<x_1, \ldots, x_m \mid r_1, \ldots, r_n\>$ be a presentation of a group $\pi$. 
The \textit{Jacobian matrix} of the presentation is the $n \times m$ matrix $J$ with $i, j$ entry equal to $\partial r_i/\partial x_j$. 
\end{definition}

\begin{remark} \label{FDC} (1) When $\<x_1, \ldots, x_m \mid r_1, \ldots, r_n\>$ is a presentation of the group $\pi=\pi_L$ of a link $L$, the following specializations of the Jacobian matrix $J$ will be used.

The abelianization $\pi/\pi'$ is a free abelian group of finite rank $\mu$ equal to the number of components of the link. An abelianization homomorphism $\a: \pi \to (t_1) \times \cdots \times (t_\mu)$ can be defined sending an oriented meridian of the $i$th component of the link to $t_i$, and we use $\a$ to identify the integral group ring $\Z[\pi/\pi']$ with the Laurent polynomial ring $\Z[t_1^{\pm1},\ldots,t_{\mu}^{\pm1}]$.
(The homomorphism $\a$ depends only on the order and orientation of link components. We follow Fox's convention \cite[p. 122]{Fo62} that the meridian $t_i$ represents a loop whose linking number with the $i$th component is $-1$.) By extension we have 
a group ring homomorphism $\a: \Z[F(x_1, \ldots, x_m)] \to \Z[\pi/\pi']$, where $F(x_1, \ldots, x_m)$ is the free group on $x_1, \ldots, x_m$. The specialization $J^\a$ is defined by replacing each entry of $J$ with its image under $\a$. 

For any group presentation of $\pi$, the matrix $J^{\alpha}$ represents a homomorphism of free $\Z[\pi/\pi']$-modules with cokernel isomorphic to the \emph{Alexander module} of $L$; that is, the first homology of the universal abelian covering space of $\S^3\setminus L$ modulo the preimage of a point. The matrix $J^\a$ is also called an \textit{Alexander matrix} of $L$. 

Let $\t$ be the composition of $\a$ with the homomorphism $\Z[\pi/\pi'] \to \Z[t^{\pm 1}]$ sending each $t_i$ to $t$. The specialization $J^\t$ is defined by replacing each entry of $J$ with its image under $\t$.


Finally, let $\nu$ be the composition of $\t$ with the homomorphism $\Z[t^{\pm 1}] \to \Z[C]$, where $C$ is the 2-element group $\{1, -1\}$, and $t$ is mapped to $-1$. We refer to the specialization as \textit{$2$-reduction}. Define $J^\nu$ by replacing each entry of $J$ with its image under $\nu$.


\medskip

(2) It is useful to define the \textit{derivative} of $r_i$ as the group ring element $\sum_j (\partial r_i/\partial x_j) x_j$. (Compare with formula (2.2) of \cite{Fo53}.) 	
\end{remark}


\begin{theorem} \label{main} Assume that $D$ is a checkerboard shaded diagram of a link $L$. Let $n$ be the number of bounded unshaded regions, and $\b= \b(D)$ the number of connected components of the checkerboard graph $\G=\G(D)$. Then $\pi_L = \pi_1(\S^3 \setminus L)$ has a presentation of the form
\begin{equation} \< x_1, \ldots, x_{n+\b} \mid r_1, \ldots, r_n \> \end{equation}
such that \begin{itemize}
\item generators $x_1, \ldots, x_n$ correspond to the bounded unshaded regions of $D$;
\item generators $x_{n+1}, \ldots, x_{n+\b}$ correspond to certain shaded regions of $D$, one for each component of $\G(D)$, with $x_{n+\b}$ corresponding to a region adjacent to the unbounded region of $D$; 
\item each relator corresponds to a distinct unshaded bounded region of $D$; and
\item the 2-reduction $J^\nu$ of $J$ is equal to $(G\  {\bf 0})$, where $G$ is the reduced Goeritz matrix and {\bf 0} is the $n \times \b$ zero matrix.
\end{itemize}
\end{theorem}

A link diagram is \textit{split} if some embedded circle separates it into two nonempty parts. Otherwise the diagram is \textit{non-split}. Since any region of $D$ can be regarded as the unbounded region via stereographic projection, the following is an immediate consequence of Theorem \ref{main}. 

\begin{cor} \label{cormain} Assume that $D$ is a non-split diagram of a link $L$. The group $\pi_L$ is generated by $m$ elements, where $m$ is the minimum of the numbers of shaded and unshaded regions of $D$.\end{cor}

The number $m$ in Corollary \ref{cormain} is often much smaller than the number of generators required for a Wirtinger presentation. For instance, the Wirtinger presentation of the group of the pretzel link $K(p,q,r)$ has $p+q+r$ generators; but, for any values of $p, q$ and $r$, the presentation of Theorem \ref{main} requires only three generators.

\section{Proof of Theorem \ref{main}} \label{proof} We review the Dehn presentation of a link group $\pi_L$ (see \cite{LS77} for further details). Begin with a diagram $D$ of $L$, a generic projection of $L$ in the plane, using an artistic device to indicate how arcs pass over each other. Let $U_0$ denote the unbounded region. Choose two basepoints, one above $U_0$, the other below $U_0$ and directly under the first. Each complementary region $R$ of the diagram has an associated element of $\pi_L$, denoted also by $R$, and represented by a loop described as follows. Begin at the upper basepoint, and follow a horizontal path to a point over the interior of $R$; then descend along a vertical path through $R$ until reaching the depth of the lower basepoint; travel along a horizontal path to the lower basepoint; finally ascend through $U_0$ to the upper basepoint. Notice that the element of $\pi_L$ corresponding to $U_0$ is the identity.


Defining relators are of the form $R_1\overline R_2 R_3 \overline R_4$ (see Figure \ref{Dehnrelator}), one for each crossing of $D$. Here and throughout we denote the inverse of a group element $g$ by $\overline g$. By a \textit{Dehn relator} we will mean any such relator. With these relators, the generators $R$ corresponding to all regions generate the free product $\hat \pi \cong  \pi_L*\Z$, where the infinite cyclic factor is generated by $U_0$. The Dehn presentation of $\pi_L$ is obtained by including the relator $U_0$ along with the Dehn relators. 



\begin{figure}
\begin{center}
\includegraphics[height=1.2 in]{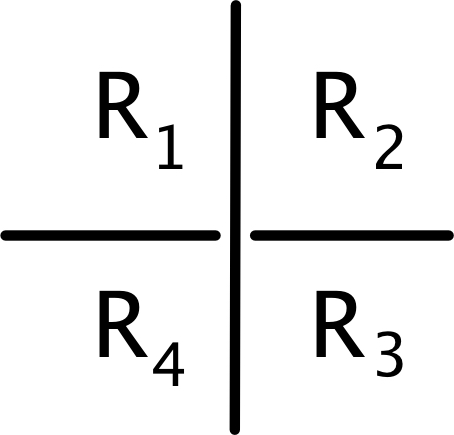}
\caption{Crossing with Dehn relator $R_1 \overline{R_2} R_3 \overline{R_4}$}
\label{Dehnrelator}
\end{center}
\end{figure}

\begin{remark}\label{Dehn} We offer a few general comments about Dehn presentations. \medskip

(1) Dehn generators have infinite order in $\hat \pi$, a fact that can be seen by mapping $\hat \pi$ to the infinite cyclic group $\Z$, sending each generator to $1$. \medskip

(2) Unlike the Wirtinger presentation, Dehn presentations do not require arcs of the link diagram to be oriented.  \medskip

(3) Re-indexing the regions $R_i$ produces equivalent presentations of $\hat \pi_L$.  (See \cite{LS77} for these and other facts about Dehn presentations.)  \medskip 

(4) 
If the overpassing arc in Figure \ref{Dehnrelator} belongs to the $j$th component of the link and is oriented upward, then $\a(R_1 \overline{R}_{2})=t_j=\a(R_4 \overline{R}_{3})$. Similarly, if the underpassing arc belongs to the $k$th component and is oriented from left to right then $\a(R_1 \overline{R}_{4})=t_k=\a(R_2 \overline{R}_{3})$.

\end{remark}


We prove Theorem \ref{main} first for any non-split link diagram $D$. For such diagrams every bounded region is homeomorphic to a disk. \bigskip

Select a shaded region adjacent to the unbounded region and label it $S_0$. It will correspond to the generator $x_{n+1}$ in the statement of the theorem. Recall that the unbounded region is labeled $U_0$.  Denote the remaining unshaded regions by $U_1, \ldots, U_n$; they will correspond to the generators $x_1, \ldots, x_n$.

Bounded faces of the plane graph $\G$ are identified with the regions $U_1, \ldots, U_n$. We orient the boundary $\partial U_i$ of each $U_i$ in the counterclockwise sense and join it to $S_0$ by a base path in $\G$. The homotopy classes of these based boundary loops 
freely generate the fundamental group $\pi_1(\G,S_0)$.

For each $U_i$ we define a relator $r_i$ in the generators $U_1, \ldots, U_n, S_0$, as follows.  Beginning at $S_0$, we follow the based loop associated to $U_i$ and use the Dehn relations corresponding to successive edges in order to rewrite shaded generators (vertices along the loop) in terms of $S_0$ and the unshaded generators of the regions that border the based loop (see Figure \ref{boundary}). Upon returning to $S_0$ we have the \textit{return value} of the based boundary loop, a word $W_i$ in $U_1, \ldots, U_n, S_0$. We define $r_i$ to be $W_i\overline{S}_0$, the \textit{boundary relator} of the region $U_i$.


\begin{figure} [bht]
\begin{center}
\includegraphics[height=1 in]{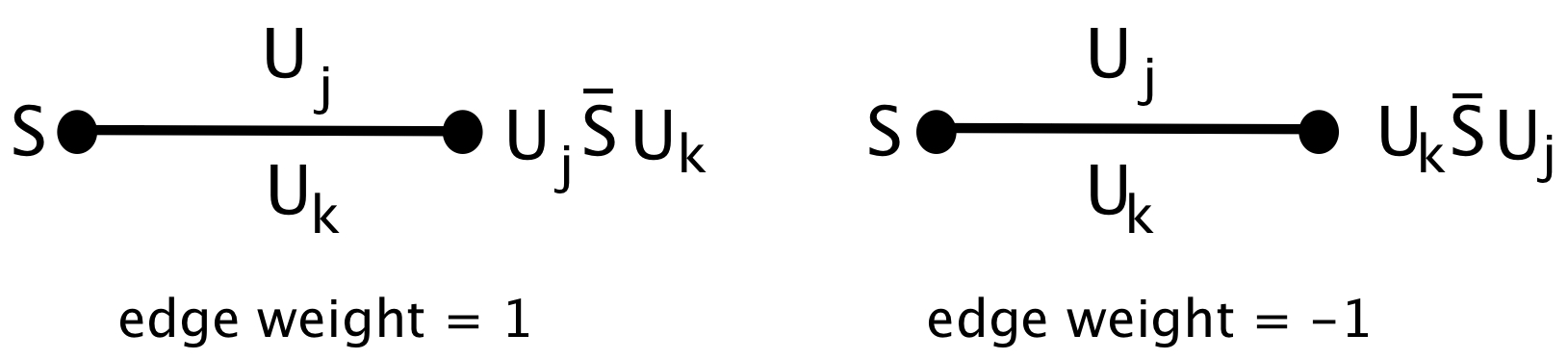}
\caption{Computing boundary along an edge of $\G$}
\label{boundary}
\end{center}
\end{figure}

By construction, each boundary relator $r_i$ is a consequence of the Dehn relators that correspond to the edges of the associated loop. It follows that adjoining $r_1, \ldots, r_n$ to the Dehn presentation does not change the fact that we have a presentation of $G$. 

Since the based boundaries of the regions $U_1, \ldots, U_n$ generate the fundamental group of $\G$, 
it follows that $W\overline{S}_0$ is in the normal closure of $r_1, \ldots, r_n$ for the boundary of {\sl any} based loop in $\G$. In particular, 
the loop that borders the unbounded region $U_0$ determines a relation $r^{\rm out}$ that is a consequence of $r_1, \ldots, r_n$. 


Computation of the return values $W$ is expedited by the following combinatorial process. 
Travel in the preferred direction along a based boundary loop of $\G$ that contains $\partial U_i$ and determines the boundary relator $r_i$. At each edge record a formal fraction $ \frac{a}{b}$, where $a, b$ are labels of the regions to either side of each edge; if the edge is weighted $+1$ (resp. $-1$), then $a$ is the label of the region to the left (resp. right) while $b$ is the label of the region to the right (resp. left). 

If the based boundary loop containing $\partial U_i$ has odd length $2l-1$, then we obtain a sequence:  

\begin{equation}\label{odd} \frac{a_1}{b_1},  \frac{a_2}{b_2},  \cdots,   \frac{a_{2l-2}}{b_{2l-2}}, \frac{a_{2l-1}}{b_{2l-1}} \end{equation}
The return value $W_i$ has the form $W'\overline S_0W''$, where 
$W'$ is the zig-zag alternating product of numerators and denominators in the sequence (\ref{odd}), working backward from the numerator $a_{2l-1}$ of the last term, and including the inverse of each denominator. $W''$ is a similar product, working forward from the denominator $b_1$ of the first term and including the inverse of each numerator. Explicitly,

\begin{equation}\label{oddrelator} W_i = a_{2l-1} \overline b_{2l-2} \cdots \overline b_2 a_1 \overline S_0
b_1 \overline a_2 \cdots \overline a_{2l-2} b_{2 l-1}. \end{equation}

If the based boundary loop containing $\partial U_i$ has even length $2l$, then we obtain a sequence:  

\begin{equation}\label{even} \frac{a_1}{b_1},  \frac{a_2}{b_2},  \cdots   \frac{a_{2l-1}}{b_{2l-1}}, \frac{a_{2l}}{b_{2l}} \end{equation}
The return value $W_i$ now has the form $W' S_0W''$, where 
$W'$ is the zig-zag alternating product of numerators and denominators in the sequence (\ref{even}), working backward from the numerator $a_{2l}$ of the last term, and including the inverse of each denominator. $W''$ is again a similar product, beginning with the inverse of the numerator $a_1$ of the first term, and including the inverse of each numerator. We have 

\begin{equation}\label{evenrelator} W_i = a_{2l} \overline b_{2l-1} \cdots a_2 \overline b_1  S_0 \overline
a_1  b_2 \cdots \overline a_{2l-1} b_{2 l}. \end{equation}

For any based boundary loop, we can replace the counterclockwise direction of $\partial U_i$ with the clockwise direction. The sequence of formal fractions that we obtain is a formal inverse: the order of terms is reversed while numerators and denominators are interchanged. Replacing $W_i$ by the new return value produces another word $r^*$ that we also call a boundary relator. It is not difficult to check that $r^*$ is a cyclic permutation of $r$ (resp. $r^{-1}$) if the loop has odd (resp. even) length.

Next we eliminate all shaded generators except $S_0$. For this it is convenient to use a spanning tree $T$ of $\G$. 
Following branches of $T$ from $S_0$, we rewrite each shaded generator (vertex) in terms of $U_1, \ldots, U_n, S_0$ (see Figure \ref{boundary}). At each step we eliminate via a Tietze transformation both a shaded generator (vertex) and a Dehn relation (incident edge). 

Each remaining Dehn relator $r$ corresponds to an edge $e$ of $\G$ not contained in the spanning tree $T$. Consider the unique based loop in $T \cup e$ with arbitrary orientation. The boundary relator of the loop is an element of the normal closure of the set of Dehn relators associated to the edges of the loop. 
All but the relator corresponding to $e$ are trivial when rewritten in terms of $U_1 \ldots, U_n, S_0$. Since the boundary relator is a consequence of $r_1, \ldots, r_n$, so is the Dehn relator $r$ corresponding to $e$. We discard it. 

It follows now that the link group $\pi_L$ has a presentation with generators $U_1 \ldots, U_n, S_0$ and boundary relators of the based boundaries of the regions $U_1, \ldots, U_n$. As discussed in Section \ref{Fox}, it follows that if $J$ is the Jacobian matrix of this presentation, then $J^{\a}$ is an Alexander matrix for $L$.


We proceed with a proof that the 2-reduced Jacobian matrix $J^\nu$ coincides with $(G\  {\bf 0})$, where $G$ is the reduced Goeritz matrix $G=G(D)$ and $\bf{0}$ is a column of zeroes. 

Recall that the rows and columns of $G$ correspond to the bounded unshaded regions $U_1, \ldots, U_n$. Each $U_i$ corresponds to a bounded region of the graph $\G$, and we can obtain the entries of the $i$th row of the reduced Goeritz matrix by following the boundary of this region in the counterclockwise direction. With respect to this counterclockwise orientation, the region $U_i$ will be on the left of each edge and an unshaded region $U_j$ will be on the right. If $i \neq j$ and the edge has positive weight, then we record $+1$ (resp. $-1$) in the $ii$ (resp. $ji$) entry. If $i \neq j$ and the edge has negative weight, then we record $-1$ (resp. $+1$) in the $ii$ (resp. $ji$) entry. 


\begin{lemma}
\label{nuimage}
The homomorphism $\tau$ applied to any shaded generator yields an odd power of $t$, while $\tau$ applied to any unshaded generator yields an even power of $t$. It follows that every shaded generator $S$ has 2-reduction $\nu(S)=-1$, while every unshaded generator $U_i$ has 2-reduction $\nu(U_i)=+1$. 
\end{lemma}
\begin{proof}The lemma is verified recursively using part (4) of Remark \ref{Dehn}, starting with $U_0=1$.
\end{proof}

Recall that the boundary relator $r_i=W_i \overline{S}_0$ is constructed by following the based boundary of $U_i$, taking into account the Dehn relation at each crossing. It follows that $W_i$ consists almost completely of unshaded generators; the only exception is a single appearance of either $S_0$ or $\overline{S}_0$. Regardless of whether it is $S_0$ or $\overline{S}_0$ that appears in $W_i$, Definition \ref{fdc} and Lemma \ref{nuimage} imply that $\nu(\partial r_i / \partial S_0)=0$.

Now consider an edge $e$ of the boundary $\partial U_i$, and give $e$ the edge direction consistent with the counterclockwise orientation of $\partial U_i$. Suppose $e$ has positive weight. Then the associated Dehn relation has the form $S' = U_i \overline{S} U_j$, where $S, S'$ are the initial and terminal vertices, respectively. It follows that the relator $r_i$ will have either the form $AU_iBS_0CU_jD \overline{S}_0$ or the form $A \overline{U_j} B \overline{S}_0C \overline{U_i} D \overline{S}_0$, where $A,B,C,D$ include only unshaded generators. Either way, Definition \ref{fdc} and Lemma \ref{nuimage} tell us that the contribution of the indicated appearance of $U_i$ to the value of $\nu(\partial r_i / \partial U_i)$ is $+1$, and the contribution of the indicated appearance of $U_j$ to the value of $\nu(\partial r_i / \partial U_j)$ is $-1$. These contributions are the same as the contributions of the edge $e$ to entries of the $i$th row of the Goeritz matrix $G$.

Similarly, if $e$ has negative weight then the relator $r_i$ will have either the form $AU_jBS_0CU_iD \overline{S}_0$ or the form $A \overline{U_i} B \overline{S}_0C \overline{U_j} D \overline{S}_0$, where $A,B,C,D$ include only unshaded generators. Once again, the contributions of the indicated appearances of $U_i$ and $U_j$ to the $i$th row of $J^{\nu}$ are equal to the contributions of $e$ to the $i$th row of $G$.

This completes the proof of Theorem \ref{main} for non-split diagrams. 


\begin{figure}
\begin{center}
\includegraphics[height=2 in]{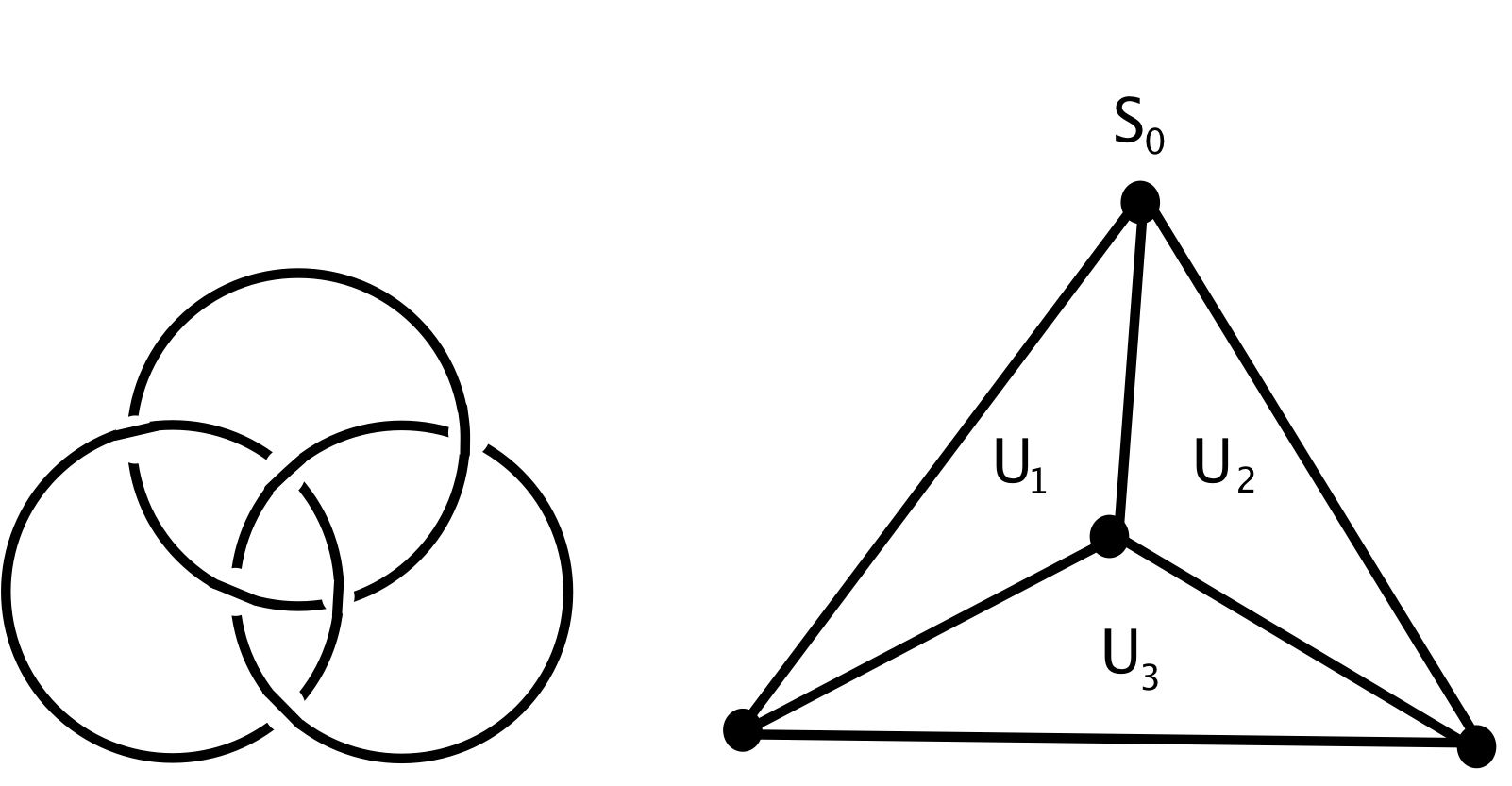}
\caption{Diagram $D$ of Borromean rings $L$ (left) and its checkerboard graph $\G$ (right)}
\label{brings}
\end{center}
\end{figure}

\begin{example} \label{BR} Consider the diagram $D$ of the Borromean rings and associated checkerboard graph $\G=\G(D)$ in Figure \ref {brings}. 
All edges have weight $+1$. The based boundary loop of $U_1$ yields the sequence of formal fractions
$\frac{U_1}{U_0}, \frac{U_1}{U_3}, \frac{U_1}{U_2}$ and the associated return value $W_1$ is the final element of the sequence: $$U_1 \overline{S}_0U_0,\quad U_1 \overline{U}_0 S_0 \overline{U}_1U_3,\quad U_1\overline{U}_3U_1 \overline{S}_0 U_0 \overline{U}_1 U_2.$$ 
\noindent The relator $r_1$ is $U_1\overline{U}_3U_1 \overline{S}_0 U_0 \overline{U}_1 U_2\overline{S}_0$. Similarly, $r_2$ and $r_3$ are, respectively, 
$$U_2 \overline{U}_3 U_2 \overline{S}_0 U_1 \overline{U}_2U_0\overline{S}_0 $$ 
$$U_1 \overline{U}_2 U_3 \overline{U}_1U_2\overline{S}_0U_1 \overline{U}_3U_0 \overline{U}_3 U_2\overline{S}_0 \overline{S}_2.$$
In order to get a  presentation for $\pi_L$ we must delete the occurrences of $U_0$. 
$$\pi_L \cong \< U_1, U_2, U_3, S_0 \mid U_1\overline{U}_3U_1 \overline{S}_0  \overline{U}_1 U_2\overline{S}_0, \   U_2 \overline{U}_3 U_2 \overline{S}_0 U_1 \overline{U}_2\overline{S}_0 , \ 
U_1 \overline{U}_2 U_3 \overline{U}_1U_2\overline{S}_0 U_1 \overline{U}_3 \overline{U}_3 U_2\overline{S}_0\rangle$$
The 2-reductions of the boundary relators of $r_1, r_2, r_3$ can also be computed using above sequences of fractions. For $r_1$ we have: 
$$U_1 -U_0, \quad U_1-U_0 + U_1 - U_3, \quad U_1-U_0 + U_1 - U_3 +U_1-U_2 = 3 U_1 - U_0-U_2-U_3.$$
Similarly, $r_2$ and $r_3$ yield, respectively, $3U_2 - U_1-U_3-U_0$ and $3U_3-U_1 -U_0-U_2$. To construct the 2-reduction of the Jacobian matrix for the presentation of $\pi_L$, we delete occurrences of $U_0$. The result is
$$\begin{pmatrix} 3 & -1 & -1 & 0  \\ -1 & 3 & -1 & 0 \\ -1 & -1 & 3 & 0 \end{pmatrix},$$ with the last column corresponding to $S_0$. This is the same $(G\  {\bf 0})$ matrix that results from the definition of $G$ (see Section \ref{goeritz}).
\end{example} 


Finally, we consider a general diagram $D$ of any link $L$. When $D$ is split, the checkerboard graph $\G=\G(D)$ is combinatorially well defined but does not contain complete information about $L$. 
(Consider, for example, unlink diagrams consisting of $\mu$ circles, some of which may be concentric.) As before, label the unbounded unshaded region $U_0$ and the remaining ones $U_1,\ldots,U_n$. For each component $\G_\lambda$ of the graph $\G$, $\lambda=0, \ldots, \beta-1$, we choose a shaded region $S_\lambda$ corresponding to a vertex of $\G_\lambda$.
We identify $U_0, U_1, \ldots, U_n, S_0, \ldots, S_{\b-1}$ and remaining shaded regions of $\Si$ with the generators of the Dehn presentation for $\hat \pi_L$ arising from the diagram $D$. 

Consider the surface $\Si$ in the plane consisting of the shaded regions of $D$ together with marked bands replacing the crossings between adjacent regions. Each marking is $+1, -1$ according to the Goeritz index of the crossing, as in Figure \ref{index}.
Each component $\Si_\lambda$ corresponds to a graph component $\G_\lambda$. 
The group $\pi_1(\Si_\lambda, S_\lambda)$ is free, generated by embedded loops based at $S_\lambda$ that run counter-clockwise around the holes which correspond to the unshaded regions $U_{\lambda,1}, \ldots, U_{\lambda, n_\lambda}$ exterior to the surface. Each loop determines a directed cycle graph with vertices corresponding to $S_\lambda$ and other shaded regions, and edges corresponding to  traversed bands. As before, we label edges with weights $+1, -1$ according to the Goeritz index of the crossing. We also label the left- and right-hand sides of each directed edge with symbols of the unshaded regions $U_j, U_k$ that appear on the those sides of the band. Then using Figure \ref{boundary} we define the \textit {return value} $W_{\lambda, i}$ of the based loop to be the word in $U_0, \ldots, U_n, S_\lambda$ obtained by following the loop around. (If the loop  avoids crossings then the return value is $1$.) We define the \textit {boundary relator} $r_{\lambda, i}$ to be $W_{\lambda, i}\overline{S_\lambda}$.  And as before, the boundary relator of {\sl any} based loop of $\Si_\lambda$ is contained in the normal closure of the $r_{\lambda, i}$.

Each component $\Si_\lambda$ has a combinatorially defined checkerboard graph $\G_\lambda$ with vertices and edges corresponding to shaded regions and crossings. We select a spanning tree $T_\lambda$ for $\G_\lambda$ and use it and Tietze transformations to eliminate shaded generators other than $S_\lambda$ as well as the Dehn relators corresponding to its edges. The same argument as in the case of non-split diagrams shows that the remaining Dehn relators, rewritten in terms of $U_1, \ldots, U_n, S_0, \ldots, S_{\b-1}$, are consequences of the boundary relators $r_{\lambda, 1}, \ldots, r_{\lambda, n_\lambda}$. We delete them from the presentation. 

We have shown that $\hat \pi_L$ has a presentation with generators $U_0, U_1, \ldots, U_n, S_0, \ldots, S_{\b-1}$ and relators $r_{0, 1}, \ldots, r_{0, n_0}, \ldots, r_{\b-1, 1}, \ldots, r_{\b-1, n_{\lambda-1}}$. Adjoining the relator $U_0$ yields a presentation for $\pi_L$. 

The unshaded regions of $D$ that are non-simply connected border different components of the surface $S$, and hence the 2-reduced Jacobian of the presentation for $\pi_L$ that we have described can differ from the reduced Goeritz matrix $G$ in the corresponding rows and columns. 

We rectify the problem by adjusting our presentation of $\pi_L$. Consider a bounded non-simply connected unshaded region. It is a simply-connected region for some component $S_{\lambda_0}$ that contains one or more components $S_{\lambda_1}, \ldots, S_{\lambda_s}$.
Replace the boundary relator $r_{\lambda_0}$ in the presentation just obtained with 
$r_{\lambda_0}\tilde r_{\lambda_1}^*\cdots \tilde r_{\lambda_s}^ *$. (The order of the relations will not matter. Recall that $\tilde r$ is the boundary relation of the outermost loop of the component, and ${}^*$ indicates that the loop in traversed in the clockwise direction.) We repeat the procedure for each bounded non-simply connected unshaded region of $D$.

That the new presentation   is equivalent to the one with which we began can easily be seen by considering the relations from innermost components of $\Si$ and working outward. Any relation that we append is a consequence of a previous relator. 

It is straightforward to see that the leading principal minor of $J^\nu$ of the new presentation coincides with the reduced Goeritz matrix.

\begin{figure}
\begin{center}
\includegraphics[height=2 in]{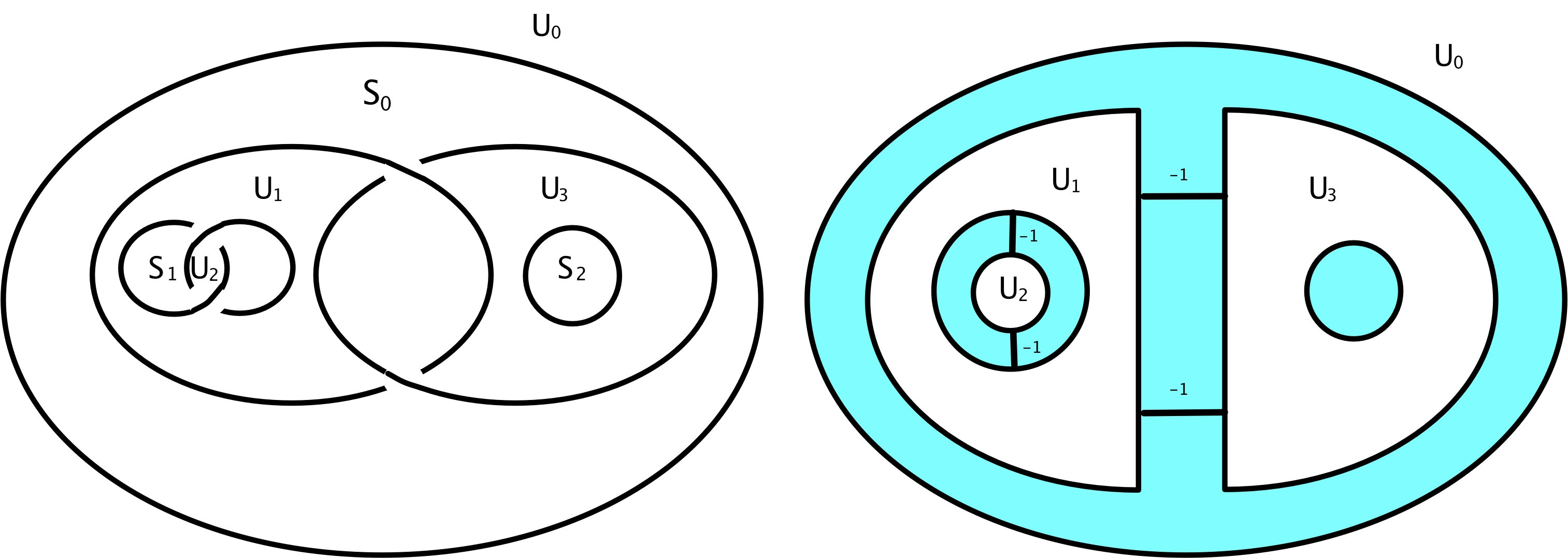}
\caption{Split diagram $D$ of link $L$ (left) and surfaces $\Si_0, \Si_1, \Si_2$ (right) }
\label{Split}
\end{center}
\end{figure}

\begin{example} Consider the split diagram $D$ of the 6-component link $L$ in Figure \ref{Split}. The surface $\Si$ has three components $\Si_0, \Si_1$ and $\Si_2$ containing shaded regions $S_0, S_1$ and $S_2$, respectively. 
The fundamental group $\pi_1(\Si_0, S_0)$ is freely generated by two based loops, one running around the left-hand side and the other along the right. Their return values are easily computed using cycle graphs, each of length two and edges with weight $-1$. 
The first boundary is $U_3 \overline{U}_1 S_0 \overline{U}_3 U_1$. 
The second is $U_1 \overline{U}_3 S_0 \overline{U}_1 U_3$.
The surface $\Si_1$ has an infinite cyclic fundamental group, and the boundary of a based loop generator is 
$U_1 \overline{U}_2 S_1 \overline{U}_1 U_2.$ The surface $\Si_2$ is simply connected, and we do not need to compute any boundary for it. Putting all of this together, we have:
$$\pi_L \cong \< U_1, U_2, U_3, S_0, S_1, S_2 \mid 
U_3 \overline{U}_1 S_0 \overline{U}_3 U_1\overline{S}_0, \ 
U_1 \overline{U}_3 S_0 \overline{U}_1 U_3\overline{S}_0, \ 
U_1 \overline{U}_2 S_1 \overline{U}_1 U_2\overline{S}_1 \>.$$
The unshaded region $U_1$ is non-simply connected. We modify its assocated relator in order to produce a presentation that will yield the reduced Goeritz matrix. 

The unshaded region $U_3$ is also non-simply connected. However, its associated relator needs no modification since the inner boundary of the region has trivial boundary. 

The modified presentation of $\pi_L$ is: 
$$\pi_L \cong \< U_1, U_2, U_3, S_0, S_1, S_2 \mid 
U_3 \overline{U}_1 S_0 \overline{U}_3 U_1\overline{S}_0U_2 \overline{U}_1 S_1 \overline{U}_2 U_1\overline{S}_1, \ 
U_1 \overline{U}_3 S_0 \overline{U}_1 U_3 \overline{S}_0,\ 
U_1 \overline{U}_2 S_1 \overline{U}_1 U_2\overline{S}_1 \>.$$

Direct computation using the free differential calculus shows that the 2-reduced version of the Jacobian matrix of the last presentation is

$$J^\nu = \begin{pmatrix} -4 & 2 & 2 & 0 &0 &0\\ 2 & 0 & -2& 0 &0 &0\\ 2 & -2 & 0& 0 &0 &0 \end{pmatrix}.$$ 

The submatrix consisting of the first three columns is the reduced Goeritz matrix $G$ associated to the shaded diagram $D$.

\end{example}

\section{Special diagrams} \label{special} Consider an oriented link $L$ with diagram $D$. As before, we checkerboard shade the diagram so that the unbounded region remains unshaded. In this section we  assume that the diagram is \textit{special}, that is, its shaded regions form an oriented spanning surface $F$ for the link. (Every link has a special diagram. See, for example, \cite{BZ03}.) Using arc orientations and the right-hand rule, we label each shaded region by  $+$ or $-$, regarding regions labeled $+$ as one side of the surface, regions labeled $-$ as the other. 

We assume also that $D$ is a non-split diagram. As in Section \ref{proof}, we consider the plane checkerboard graph $\G$, identifying its vertices with shaded regions of $D$ and its bounded faces with the bounded unshaded regions $U_1, \ldots, U_n$ of $D$.

Let $S_0$ be a shaded vertex labeled $+$. For each $U_i$ we select  a base path from $S_0$ to a vertex labeled $+$ on the boundary $\partial U_i$. By Theorem \ref{main} the link group $\pi_L$ has a presentation of the form 
\begin{equation}\label{pr1}\<S_0, U_1, \ldots, U_n \mid W_1 \overline S_0, \ldots, W_n\overline S_0 \>.
\end{equation}

Since $D$ is special, the length of every boundary $\partial U_i$ is even. The return value $W_i$ has the form $A_iS_0\overline{B}_i$, where $A_i, B_i$ are words in $U_1, \ldots, U_n$, each having even length, and the presentation (\ref{pr1}) can be rewritten as: 

\begin{equation}\label{pr2}\<S_0, U_1, \ldots, U_n \mid \overline S_0A_1S_0=B_1, \ldots, \overline S_0 A_n S_0=B_n \>.
\end{equation} 

As in Section \ref{proof} the words $A_i, B_i$ can be read from the sequence of formal fractions recorded as we travel along the based boundary of $U_i$. Regard the $A_i, B_i$ as words in the free group $F_n$ on the generating set $\{U_1, \ldots, U_n\}$, and
let $U_1^{a_{i,1}}\cdots U_n^{a_{i,n}}$ be the image of $A_i$ in the abelianization $F_n/F'_n$.  Define $A$ to be the integral $n \times n$ matrix $(a_{i.j})$. Define $B$ similarly as $(b_{i,j})$. 

Consider the Seifert matrix $H^+$ with $i,j$-entry equal to the linking number $Lk(\partial U_i, \partial U_j^+)$. Here we regard $U_i, U_j$ as oriented curves in the surface $F$, and $\partial U_j^+$ as a
copy of $\partial U_j$ pushed off the surface in the direction of the positive normal vector. Contributions to linking numbers by base paths cancel and so we ignore base paths. We consider also the Seifert matrix  $H^-$ similarly defined but
with $i,j$-entry $Lk(\partial U_i, \partial U_j^-)$, where $\partial U_j^-$ is obtained by pushing off in the negative normal direction.

\begin{theorem} \label{seifert} Let $D$ be a non-split, special diagram of a link $L$. Then $\pi_L$ 
has a presentation of the form $\<S_0, U_1, \ldots, U_n \mid \overline S_0A_1S_0=B_1, \ldots, \overline S_0 A_nS_0=B_n \>,$ where the matrix $B$ (resp. $A$) is equal to the Seifert matrix $H^+$ (resp. $H^-$) of the diagram $D$. If additionally $D$ is alternating, then 
the presentation describes $\pi_L$ as an HNN extension with stable letter $S_0$. 
\end{theorem}

\begin{figure}
\begin{center}
\includegraphics[height=3 in]{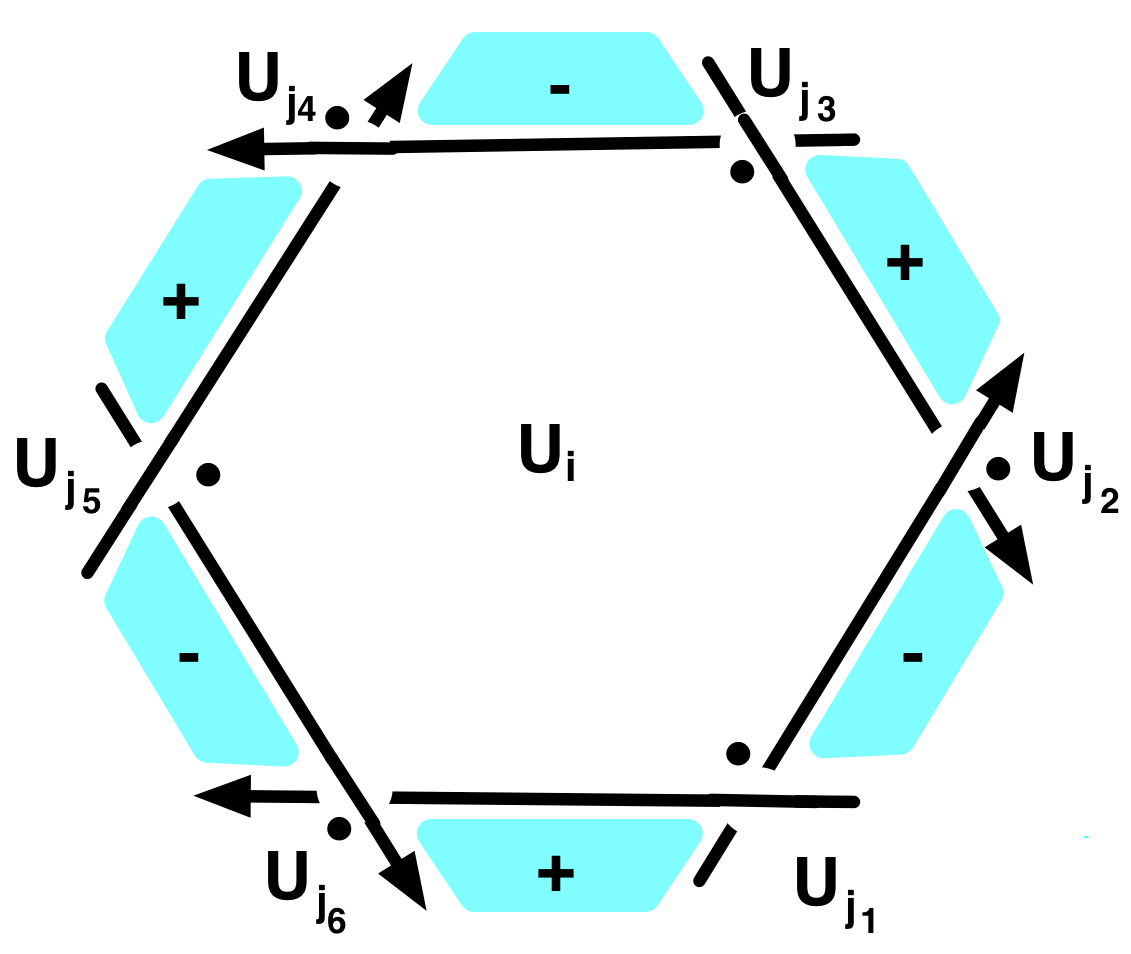}
\caption{Unshaded region $U_i$ of special alternating diagram }
\label{Seifert}
\end{center}
\end{figure}

\begin{proof} The Seifert matrix $H^+$ can be computed directly from the diagram $D$. Begin by placing a dot in the corners of unshaded regions if they appear on the left of an under-crossing arc
with respect to its preferred orientation, as illustrated in Figure \ref{Seifert}. At each crossing  $c$ of the diagram a dot will appear in exactly one unshaded region. Define 
$\e_i(c)=1$ if the dot appears in $U_i$, zero otherwise. We write $c \in \partial U$ if the crossing is incident to the region $U$. Then  
\begin{equation}\label{smatrix} H^+_{i,i}= \sum_{c \in \partial U_i} \eta(c) \e_i(c), \quad H^+_{i,j}=
-\sum_{c \in \partial U_i \cap \partial U_j} \eta(c)\e_j(c), \ ( i\ne j). \end{equation}
(See page 231 of  \cite{BZ03}. The reader is warned that the second summation there is missing the negative sign. The proof, however, is correct.) 

We can use formulas (\ref{smatrix}) to find the $i$th row of the Seifert matrix $H^+$ by imagining that we are standing in the center of $U_i$. The diagonal term $H^+_{i,i}$ is  the number of dotted corners that we see, each weighted by the Goeritz index $\eta(c)=\pm 1$ of the nearby crossing. Each undotted corner, diagonally across from some region $U_j$, contributes $\eta(c)$ to the $j$th column. 
In Figure \ref{Seifert}, for example, where all Goeritz indices are $1$, we have $H^+_{i,i}=3$ and $H^+_{i, j_k} = -1$ if $k = 2,4,6$; other entries $H^+_{i,j}$ are zero. 

Now consider the based boundary loop of $U_i$. First assume that all Goeritz indices are $1$. Beginning at a $+$ vertex and traveling around the loop, we record a sequence of formal fractions 
\begin{equation} \frac{\dot U_i}{U_{j_1}}\frac{U_i}{\dot U_{j_2}} \cdots \frac{\dot U_i}{U_{j_{2l-1}}}\frac{U_i}{\dot U_{j_{2l}}},    \end{equation}
where $l$ is the length of $\partial U_i$ and $\cdot$ indicates that a dot is found in that region of the corner.  The word $A_i$ is the zig-zag alternating product of numerators and denominators, beginning with the numerator of the last term: 
$$A_i= U_i\overline{U}_{j_{2l-1}}U_i\overline{U}_{j_{2l-3}}\cdots U_i\overline{U}_{j_1}.$$
Likewise, $B_i$ is the zig-zag alternating product of numerators and denominators, beginning with the inverse of the denominator of the {\sl last} term: 
$$B_i= \overline{U}_{j_{2l}}U_i\overline{U}_{j_{2l-2}}U_i\cdots \overline{U}_2 U_i.$$

When we construct $B_i$ from the based boundary of $U_i$, each crossing from the base path is encountered twice, once before the based boundary loop traverses $\partial U_i$ and once after. The two encounters are in opposite directions, so according to formula (\ref{evenrelator}), the two encounters contribute opposite powers of the same $U_j$ to the word $B_i$. It follows that the contributions from the base path to $B_i$ cancel in the abelianization of the free group $F_n$, and we see immediately that the contributions to the $i$th row of $H^+$ agree with those given by formulas (\ref{smatrix}). 

We have considered only diagrams with crossings having Goeritz index 1. Changing a crossing flips the corresponding numerator and denominator (but leaves the dot in place). It is easy to see that the two methods of computation continue to agree. Hence $B$ is equal to the Seifert matrix $H^+$. 

If we reverse the orientation of the diagram $D$, then the new Seifert matrix that we obtain 
is $H^-$. It is equal to transpose of $H^+$.  
The effect on the sequence of formal fractions arising from the checkerboard graph is to 
move each dot, from numerator to denominator or vice versa. We see that $A$ is equal to 
$H^-$. 

This completes the proof of the first statement of Theorem \ref{seifert}. 
For a proof of the second statement assume that $D$ is a special alternating diagram. (A special diagram is alternating if and only if all Goeritz indices have the same value.) The sets $\{A_1, \ldots, A_n\}$ and $\{B_1, \ldots, B_n\}$ generate subgroups $gp(A_1, \ldots, A_n)$ and $gp(B_1, \ldots, B_n)$ of the free group $F_n$, respectively. In order to prove that the presentation (\ref{BR}) expresses $\pi_L$ as an HNN extension with stable letter $S_0$, we must show that the homomorphism $\phi$ taking $A_i$ to $B_i$, for each $i$, is an isomorphism. It suffices to show that $A_1, \ldots, A_n$ and $B_1, \ldots, B_n$ freely generate $gp(A_1, \ldots, A_n)$ and $gp(B_1, \ldots, B_n)$, respectively. 

Since $D$ is a special alternating diagram, the determinants of $H^+$ and $H^-$ are nonzero (see Prop. 13.24 of \cite{BZ03}). Hence $A_1, \ldots, A_n$ generate a subgroup of 
$F_n/F'_n\cong \Z^n$ with finite index (equal to the absolute value of the determinant). It follows that $A_1, \ldots, A_n$ must freely generate  $gp(A_1, \ldots, A_n)$. The same argument applies to $B_1, \ldots, B_n$. 
\end{proof}

Here is a direct consequence of Theorem \ref{seifert}, analogous to the last assertion of Theorem \ref{main}.
\begin{cor} Let $D$ be a non-split, special diagram of a link $L$, let $H^+,H^-$ be the corresponding Seifert matrices defined above, and let $\widetilde H=H^- - tH^+$. If $\tau$ is the map defined in Remark \ref{Fox} then $L$ has an Alexander matrix $J$ such that $J^{\tau}=(\widetilde H \  {\bf 0})$, where $\bf 0$ is a column of zeroes.
\end{cor}

\begin{proof}
Theorem \ref{seifert} tells us that $\pi_L$ has a presentation 
$\<U_1, \ldots, U_n, S_0 \mid r_1, \ldots, r_n \>$, where $r_i=A_iS_0 \overline B_i \overline S_0$. Since $D$ is special and $U_0=1$, part (4) of Remark \ref{Dehn} tells us that every Dehn generator $U$ corresponding to an unshaded region has $\tau(U)=1$. Also, the Dehn generator $S$ corresponding to a shaded region labeled $+$ (resp. $-$) has $\tau(S)=t$ (resp. $\tau(S)= t^{-1}$). In particular, $\tau(S_0)=t$. 

Now, let $J$ be the Alexander matrix obtained from the presentation $\<U_1, \ldots, U_n, S_0\mid r_1, \ldots, r_n \>$ using the free differential calculus, as in Section \ref{Fox}. For $1 \leq i \leq n$ the image under $\tau$ of the $i$th entry of the last column (the column corresponding to $S_0$) is $\tau(A_i) - \tau(A_i S_0 \overline B_i \overline S_0)= 0$. The fact that the first $n$ columns of $J^{\tau}$ are the same as the columns of $\widetilde H$ follows from the equalities $A=H^-,B=H^+$ of Theorem \ref{seifert}.
\end{proof}

\begin{remark} More can be said about a link with a special non-split alternating diagram. It is known that the Seifert surface $F$ formed by its shaded regions has minimal genus, and splitting $\S^3$ along $F$ produces a handlebody ${\cal H}$ of genus $n$ \cite{Me84}. The boundary of ${\cal H}$ contains two copies $F_+, F_-$ of $F$ with $F_+ \cap F_- = L$. The infinite cyclic cover of $\S^3 \setminus L$ corresponding to the homomorphism $\tau:\pi_1(\S^3\setminus L) \to \Z \cong \<t\mid \>$ sending
each oriented meridian of $L$ to $t$ can be constructed by gluing countably many copies ${\cal H}_\nu$ of the handlebody end-to-end, matching $F_+\subset {\cal H}_\nu$ with $F_- \subset {\cal H}_{\nu +1}$. With appropriate choice of basepoint and base paths the gluing map induces a monomorphism of fundamental groups that corresponds to the HNN amalgamation map $\phi$ in the proof of 
Theorem \ref{seifert}.
\end{remark}

%
%
%

\bigskip

\ni Department of Mathematics\\
\ni Lafayette College\\Easton PA 18042\\
\ni Email: traldil@lafayette.edu\\ 

\ni Department of Mathematics and Statistics,\\
\ni University of South Alabama\\ Mobile, AL 36688 USA\\
\ni Email: silver@southalabama.edu\\ swilliam@southalabama.edu
\end{document}